\documentclass[11pt]{amsart}
\usepackage{mathrsfs,cleveref}
\usepackage{xcolor}

\title{Superpolynomial identities of finite-dimensional simple algebras}
\author{Yuri Bahturin}
\thanks{The first author is supported by NSERC Discovery grant \#227060-19}
\address{Department of Mathematics and Statistics, Memorial University of Newfoundland, St. John's, NL, A1C5S7, Canada.}
\email{bahturin@mun.ca}
\author{Felipe Yukihide Yasumura}
\address{Department of Mathematics, Instituto de Matem\'atica e Estat\'istica, Universidade de S\~ao Paulo, SP, Brazil}
\email{fyyasumura@ime.usp.br}
\thanks{The second named author is supported by Fapesp, grants no.~2023/03922-8 and no.~2018/23690-6}

\newtheorem{Thm}{Theorem}
\newtheorem{Lemma}[Thm]{Lemma}
\newtheorem{Prop}[Thm]{Proposition}
\newtheorem{Cor}[Thm]{Corollary}
\theoremstyle{remark}
\newtheorem{Remark}[Thm]{Remark}

\newtheorem*{Problem}{Problem}
\theoremstyle{definition}

\newtheorem{Def}[Thm]{Definition}

\begin{document}
\begin{abstract}
We investigate the Grassmann envelope (of finite rank) of a finite-dimensional $\mathbb{Z}_2$-graded algebra. As a result, we describe the polynomial identities of $G_1(\mathcal{A})$, where $G_1$ stands for the Grassmann algebra with $1$ generator, and $\mathcal{A}$ is a $\mathbb{Z}_2$-graded-simple associative algebra. We also classify the conditions under which two associative $\mathbb{Z}_2$-graded-simple algebras share the same set of superpolynomial identities, i.e., the polynomial identities of its Grassmann envelope (in particular, of finite rank). Moreover, we extend the construction of the Grassmann envelope for the context of $\Omega$-algebras and prove some of its properties. Lastly, we give a description of $\mathbb{Z}_2$-graded-simple $\Omega$-algebras.
\end{abstract}
\maketitle

\section{Introduction}
A remarkable result due to Kemer (see \cite{Kemerbook}) states that, over a field of characteristic zero, every associative PI-algebra is PI-equivalent to the Grassmann envelope of a finite-dimensional associative $\mathbb{Z}_2$-graded algebra. This result, known as Representability Theorem, has profound and important consequences. For instance, Kemer proved that every variety of associative algebras over a field of characteristic zero satisfies the Specht property \cite{Kemerbook}. Another remarkable result is due to Giambruno and Zaicev, and they establish the existence and integrity of the exponent of an associative PI-algebra \cite{GiZabook}.

On the other hand, it is known that polynomial identities completely determine a finite-dimensional prime $\Omega$-algebra over an algebraically closed field \cite{Razmyslovbook} (see also \cite{BY19a,BY19b} and references therein). So, it seems to be interesting to investigate a related question, proposed by I.~Shestakov: do the superpolynomial identities (ie.~the polynomial identities of the Grassmann envelope of a $\mathbb{Z}_2$-graded algebra) completely determine a finite-dimensional graded-simple algebra over an algebraically closed field?

In this paper, we study the Grassmann envelope of finite-dimensional graded-simple associative algebras, with a particular focus on the finitely generated Grassmann algebras. Some of our results are extended to the nonassociative setting. Firstly, we establish a general structural result regarding the finite-dimensional Grassmann envelope of an algebra (see \Cref{structure}). These results are straightforward and follow from simple observations. Next, we determine the polynomial identities of $G_1(\mathcal{A})$, where $\mathcal{A}$ is a finite-dimensional $\mathbb{Z}_2$-graded-simple associative algebra (\Cref{superPI_mnfz} and \Cref{superPI_mkl}). The first one is self-evident, while the second one utilizes the theory developed by Giambruno and Zaicev. Lastly, we classify when two finite-dimensional graded-simple associative algebras satisfy the same set of superpolynomial identities (\Cref{superpol_assoc}). We develop a general theory of the Grassmann envelope for arbitrary $\mathbb{Z}_2$-graded $\Omega$-algebras (\Cref{omega_alg}). Some of the results proved in the associative case carry over to the general setting of an arbitrary signature. Finally, we provide a description of $\mathbb{Z}_2$-graded-simple $\Omega$-algebras (\Cref{gradedsimplealg}).

\section{Preliminaries and Notation}
We assume that $\mathbb{F}$ is an algebraically closed field of characteristic zero. The Grassmann algebra is denoted by $G$ and its generatores are $e_1$, $e_2$, \dots, and $G_m$ stands for the Grassmann algebra generated by the first $m$ elements. We denote by $G^\ast$ and $G_m^\ast$ the respective Grassmann algebras without unity. It is clear that we have a decomposition $G_m=\left(G_m\right)^0\oplus\left(G_m\right)^1$, $G_m^\ast=\left(G_m^\ast\right)^0\oplus\left(G_m^\ast\right)^1$, etc, as sum of elements of even and odd lengths, providing $\mathbb{Z}_2$-graded algebras. If $\mathcal{A}=\mathcal{A}^0\oplus\mathcal{A}^1$ is a $\mathbb{Z}_2$-graded algebra, we define its Grassmann envelope as $G(\mathcal{A})=G^0\otimes\mathcal{A}^0\oplus G^1\otimes\mathcal{A}^1$. In the same way we define $G^\ast(\mathcal{A})$, $G_m(\mathcal{A})$ and $G_m^\ast(\mathcal{A})$.

We denote by $M_{k,\ell}$ the $\mathbb{Z}_2$-graded algebra $(M_{k+\ell},(\underbrace{0,\ldots,0}_\text{$k$ times},\underbrace{1,\ldots,1}_\text{$\ell$ times}))$, and we assume that $k\ge\ell\ge0$ and $k>0$ (the case $\ell=0$ corresponds to the the matrix algebra $M_k(\mathbb{F})$ endowed with the trivial grading). That is,
$$
M_{k,\ell}=\left\{\left(\begin{array}{cc}A&B\\C&D\end{array}\right)\mid A\in M_k, D\in M_\ell, B\in M_{k\times\ell}, C\in M_{\ell\times k}\right\},
$$
and the $\mathbb{Z}_2$-decomposition is
$$
(M_{k,\ell})^0=\left\{\left(\begin{array}{cc}A&0\\0&D\end{array}\right)\right\},\quad (M_{k,\ell})^1=\left\{\left(\begin{array}{cc}0&B\\C&0\end{array}\right)\right\}.
$$

The algebra $M_n(\mathbb{F}\mathbb{Z}_2)=M_n(\mathbb{F})\oplus\varepsilon M_n(\mathbb{F})$, where $\varepsilon^2=1$, has a natural $\mathbb{Z}_2$-grading. The $\mathbb{Z}_2$-grading is given by $(M_n(\mathbb{FZ}_2))^0=M_n(\mathbb{F})$ and $(M_n(\mathbb{FZ}_2))^1=\varepsilon M_n(\mathbb{F})$. It is known that a finite-dimensional $\mathbb{Z}_2$-graded-simple associative algebra over $\mathbb{F}$ is either graded-isomorphic to $M_{k,\ell}$ or $M_n(\mathbb{FZ}_2)$.

The set of polynomial identities of a given algebra $\mathcal{A}$ is denoted by $\mathrm{Id}(\mathcal{A})$. If $\mathcal{A}$ is $\mathbb{Z}_2$-graded, then the set of its $\mathbb{Z}_2$-graded polynomial identities is denoted by $\mathrm{Id}_2(\mathcal{A})$.

Let $\mathcal{A}$ and $\mathcal{B}$ be $\mathbb{Z}_2$-graded algebras. We denote $\mathcal{A}\cong_2\mathcal{B}$ if they are isomorphic as $\mathbb{Z}_2$-graded algebras.

\section{Main results: Associative case}
\subsection{Structure results on $G_m(\mathcal{A})$\label{structure}} First, note that $G_m^\ast(\mathcal{A})$ is an ideal of $G_m(\mathcal{A})$ and $G_m^\ast(\mathcal{A})^{m+1}=0$. In particular, if $m>0$ and $\mathcal{A}^1\ne0$, then $G_m(\mathcal{A})$ is never a semiprime algebra ($G(\mathcal{A})$ is not semiprime as well). Thus, we have the Wedderburn-Malcev decomposition of $G_m(\mathcal{A})$:
\begin{Prop}\label{prop1}
Let $\mathcal{A}=\mathcal{A}^0\oplus\mathcal{A}^1$ be a finite-dimensional associative $\mathbb{Z}_2$-graded algebra. Assume that $\mathcal{A}^0=\mathcal{S}+J$, where $\mathcal{S}$ is a semisimple subalgebra, $J=J(\mathcal{A}^0)$ is the Jacobson radical of $\mathcal{A}^0$, and $\mathcal{S}\cap J=0$. Then, for any $m\in\mathbb{N}$,
$$
G_m(\mathcal{A})=\mathcal{S}+(J+G_m^\ast(\mathcal{A})),
$$
and $J+G_m^\ast(\mathcal{A})$ is the Jacobson radical of $G_m(\mathcal{A})$.
\end{Prop}
\begin{proof}
Let $N=J+G_m^\ast(\mathcal{A})$. Since $\mathcal{A}^0$ is finite-dimensional, $J$ is a nilpotent ideal of $\mathcal{A}^0$. In addition, as discussed before, $G_m^\ast(\mathcal{A})$ is a nilpotent ideal as well. Thus, $N$ is a nilpotent ideal, so it is contained in the Jacobson radical of $G_m(\mathcal{A})$. Since $G_m(\mathcal{A})/N\cong\mathcal{S}$ is semisimple, we see that $J(G_m(\mathcal{A}))\subseteq N$. Thus, we obtain the equality.
\end{proof}

Applying the result for the graded-simple associative case, the previous theorem reads as:
\begin{Cor}\label{decomp}
Let $m\in\mathbb{N}$. Then, the Wedderburn-Malcev decomposition of the finite-dimensional Grassmann envelope of the finite-dimensional graded-simple associative algebras are the following:
$$
G_m(M_{k,\ell})=M_k\oplus M_\ell+G_m^\ast(M_{k,\ell}),$$
$$G_m(M_n(\mathbb{FZ}_2))=M_n+G_m^\ast(M_n(\mathbb{FZ}_2)),
$$
where $G_m^\ast(M_{k,\ell})$ and $G_m^\ast(M_n(\mathbb{FZ}_2))$ are their respective Jacobson radicals.
\end{Cor}
\begin{proof}
We have $(M_{k,\ell})^0=M_k\oplus M_\ell$ and $(M_n(\mathbb{FZ}_2))^0=M_n(\mathbb{F})$. Hence, both are semisimple and we apply \Cref{prop1}.
\end{proof}

Thus, we are able to compute the PI-exponent of such algebras:
\begin{Cor}\label{exp}
Given $m\in\mathbb{N}$, one has
$$
\mathrm{exp}(G_m(M_{k,\ell}))=k^2+\ell^2,\quad\mathrm{exp}(G_m(M_n(\mathbb{FZ}_2)))=n^2.
$$
In addition,
$$
\mathrm{exp}(G(M_{k,\ell}))=(k+\ell)^2,\quad\mathrm{exp}(G(M_n(\mathbb{FZ}_2)))=2n^2.
$$
\end{Cor}
\begin{proof}
The infinite-dimensional case is well-known (see \cite{GiZabook}). Now, from Giambruno-Zaicev results, it follows that the exponent of $G_m(M_n(\mathbb{FZ}_2))$ coincides with the dimension of its semisimple part. Thus,
$$
\mathrm{exp}(G_m(M_n(\mathbb{FZ}_2)))=\dim M_n=n^2.
$$
Now, if $\ell>0$, then it is easy to see that $M_kG_m^\ast(M_{k,\ell}) M_\ell\ne0$. Therefore, again from Giambruno-Zaicev results, one has
$$
\mathrm{exp}(G_m(M_{k,\ell}))=\dim M_k\oplus M_\ell=k^2+\ell^2.
$$
\end{proof}
As an easy consequence, the numbers $\mathrm{exp}(\lim\limits_\to G_m(\mathcal{A}))$ and $\lim_{m\to\infty}\mathrm{exp}(G_m(\mathcal{A}))$ do not have to coincide.

\subsection{Polynomial identities of $G_1(\mathcal{A})$} We describe the polynomial identities of $G_1(\mathcal{A})$ when $\mathcal{A}$ is a finite-dimensional $\mathbb{Z}_2$-graded-simple associative algebra.

The first family of algebras is an easy remark:
\begin{Prop}\label{superPI_mnfz}
$\mathrm{Id}(G_1(M_n(\mathbb{FZ}_2)))=\mathrm{Id}(M_n)$.
\end{Prop}
\begin{proof}
Note that $G_1(M_n(\mathbb{FZ}_2))=M_n(G_1)=M_n(\mathbb{F})\otimes G_1$. Since $G_1$ is a unital associative and commutative algebra and $\mathbb{F}$ is infinite, then $\mathrm{Id}(M_n(\mathbb{F})\otimes G_1)=\mathrm{Id}(M_n)$.
\end{proof}

For the next family, we have:
\begin{Thm}\label{superPI_mkl}
$\mathrm{Id}(G_1(M_{k,\ell}))=\mathrm{Id}(M_k)\mathrm{Id}(M_\ell)$.
\end{Thm}
\begin{proof}
From \Cref{decomp}, one has
$$
G_1(M_{k,\ell})=M_k\oplus M_\ell+J,
$$
where $J^2=0$. Thus, clearly $\mathrm{Id}(M_k)\mathrm{Id}(M_\ell)\subseteq\mathrm{Id}(G_1(M_{k,\ell}))$. From Giambruno-Zaicev results, it is known that
$$
\mathrm{Id}(M_k)\mathrm{Id}(M_\ell)=\mathrm{Id}(\mathrm{UT}(k,\ell)),
$$
where $\mathrm{UT}(k,\ell)$ is the upper block-triangular matrix algebra. Moreover, again from Giambruno-Zaicev result, $\mathrm{UT}(k,\ell)$ generates a minimal variety of exponent $k^2+\ell^2$. Since this number equals the exponent of $G_1(M_{k,\ell})$ (\Cref{exp}), it follows that
$$
\mathrm{Id}(M_k)\mathrm{Id}(M_\ell)=\mathrm{Id}(G_1(M_{k,\ell})).
$$
\end{proof}

\noindent\textbf{Question.} Describe the polynomial identities of $G_m(\mathcal{A})$ when $\mathcal{A}$ is an associative $\mathbb{Z}_2$-graded-simple algebra.

\subsection{Distinguishing simple algebras by their superpolynomial identities} We use the notation $G_\infty(\mathcal{A}):= G(\mathcal{A})$. First, note that, if $\mathcal{A}$ is ungraded, that is $\mathcal{A}^0=\mathcal{A}$ and $\mathcal{A}^1=0$, then for any $m\in\mathbb{N}\cup\{\infty\}$, one has
$$
G_m(\mathcal{A})=(G_m)^0\otimes\mathcal{A}.
$$
Since $(G_m)^0$ is a unital associative and commutative algebra and $\mathbb{F}$ is infinite, then $\mathrm{Id}(\mathcal{A})=\mathrm{Id}(G_m(\mathcal{A}))$. Thus, combining with \Cref{superPI_mnfz}, we have our first remark:
\begin{Lemma}\label{g1_ordinary}
For any $m\in\mathbb{N}\cup\{\infty\}$,
$$
\mathrm{Id}(G_1(M_n(\mathbb{FZ}_2)))=\mathrm{Id}(G_m(M_n(\mathbb{F}))),
$$
but $M_n(\mathbb{FZ}_2)\not\cong M_n(\mathbb{F})$.\qed
\end{Lemma}

Now, we shall investigate the algebras $M_{k,\ell}$.
\begin{Lemma}\label{nonidentity_mkl}
Let $m\in\mathbb{N}$. Then $\mathrm{Id}(M_k)^{m+1}\subseteq\mathrm{Id}(G_m(M_{k,\ell}))$. In addition, the product of $m$ standard polynomials of degree $2k$ is not a PI for $M_{k,\ell}$, that is,
$$
s_{2k}(x_1,\ldots,x_{2k})s_{2k}(x_{2k+1},\ldots,x_{4k})\cdots s_{2k}(x_{2k(m-1)+1},\ldots,x_{2km})\notin\mathrm{Id}(G_m(M_{k,\ell})).
$$
\end{Lemma}
\begin{proof}
The first assertion follows from \Cref{decomp}, since
$$
\mathrm{Id}(M_k)\subseteq\mathrm{Id}(G_m(M_{k,\ell})/G_m^\ast(M_{k,\ell})), 
$$
and $G_m^\ast(M_{k,\ell})^{m+1}=0$. To prove the last assertion, note that
$$
s_{2k}(e_{11},e_{12},\ldots,e_{k-1,k},e_{kk},e_ie_{k,k+1})=e_ie_{1,k+1},
$$
$$
s_{2k}(e_ie_{k+1,k},e_{kk},e_{k,k-1},\ldots,e_{21},e_{11})=e_ie_{k+1,1},
$$
Thus, we may find an evaluation on the product of $m$ standard polynomials giving
$$
(e_1e_{1,k+1})(e_2e_{k+1,1})\cdots(e_me_{t,t'})=e_1\cdots e_m e_{1,t'}\ne0,
$$
where $(t,t')=(1,k+1)$ if $m$ is odd, or $(t,t')=(k+1,1)$ otherwise.
\end{proof}

\begin{Lemma}\label{casemkl}
Let $m,m'\in\mathbb{N}\cup\{\infty\}$. If $\mathrm{Id}(G_m(M_{k,\ell}))=\mathrm{Id}(G_{m'}(M_{k',\ell'}))$, then $m=m'$, $k=k'$ and $\ell=\ell'$.
\end{Lemma}
\begin{proof}
Assume that $\mathrm{Id}(G_m(M_{k,\ell}))=\mathrm{Id}(G_{m'}(M_{k',\ell'}))$. If $m<\infty$, from \Cref{nonidentity_mkl}, $\underbrace{\mathrm{Cap}_{k^2+1}\cdots\mathrm{Cap}_{k^2+1}}_\text{$m+1$ times}\in\mathrm{Id}(G_m(M_{k,\ell}))$. Thus $k'\le k$ and $m'<\infty$. Reversing the argument, we see that $k=k'$. Hence, from \Cref{nonidentity_mkl}, since $\mathrm{Id}(M_k)^{m}\not\subseteq\mathrm{Id}(G_{m'}(M_{k',\ell'}))$, we obtain that $m\le m'$. Repeating the argument, one obtains $m=m'$. Finally, from \Cref{exp}, one has
$$
k^2+\ell^2=\mathrm{exp}(G_m(M_{k,\ell}))=\mathrm{exp}(G_{m'}(M_{k',\ell'}))=k^{\prime2}+\ell^{\prime2}.
$$
Thus, $\ell=\ell'$. The same conclusion holds valid if we start assuming that $m'<\infty$.

So we may assume that $m=m'=\infty$. However, this case follows from the classical Kemer's theory.
\end{proof}

Finally, we shall deal with the situation of $M_n(\mathbb{FZ}_2)$.
\begin{Lemma}\label{casemnfz2}
We have $\lfloor\frac{m}2\rfloor=\lfloor\frac{m'}2\rfloor$ if, and only if, $\mathrm{Id}(G_m(M_n(\mathbb{FZ}_2)))=\mathrm{Id}(G_{m'}(M_n(\mathbb{FZ}_2))$.
\end{Lemma}
\begin{proof}
Since $G_m(M_n(\mathbb{FZ}_2))\cong M_n(G_m)$, and $G_{2m+1}$ is PI-equivalent to $G_{2m}$, we obtain from \cite[Theorem 7.2.3]{AGPR} that
$$
\mathrm{Id}(G_{2m+1}(M_n(\mathbb{FZ}_2)))=\mathrm{Id}(G_{2m}(M_n(\mathbb{FZ}_2)).
$$
Since $\mathrm{char}\,\mathbb{F}=0$, Frenkel \cite[Theorem 5 and Proposition 6]{Frenkel} found the minimal degree of a Capelli identity satisfied by $G_m(M_n(\mathbb{FZ}_2))$, which is $n^2+\lfloor m/2\rfloor+1$. This proves the converse of the statement.
\end{proof}

Now, we shall prove that $M_{k,\ell}$ and $M_n(\mathbb{FZ}_2)$ do not satisfy the same set of superpolynomial identities, with the exception of the case described in \Cref{g1_ordinary}. We start with a well-known result:
\begin{Lemma}\label{nocapelli}
The algebras $G(M_n(\mathbb{FZ}_2))$ and $G(M_{k,\ell})$ (when $\ell>0$) satisfy no Capelli identity.
\end{Lemma}

\begin{Lemma}\label{distinctalg}
Let $m,m'\in\mathbb{N}\cup\{\infty\}$. If $\mathrm{Id}(G_m(M_{k,\ell}))=\mathrm{Id}(G_{m'}(M_n(\mathbb{FZ}_2))$, then $\ell=0$ and $m'=1$.
\end{Lemma}
\begin{proof}
First, assume that $m=\infty$. Then
$$
(k+\ell)^2=\mathrm{exp}(G(M_{k,\ell}))=\mathrm{exp}(G_{m'}(M_n(\mathbb{FZ}_2)))=\left\{\begin{array}{cc}n^2,&\text{ if $m'\in\mathbb{N}$},\\2n^2,&\text{ if $m'=\infty$}\end{array}\right..
$$
The equality $(k+\ell)^2=2n^2$ is impossible. Thus, $m'\in\mathbb{N}$ and $k+\ell=n$. So $\dim G_{m'}(M_n(\mathbb{FZ}_2))<\infty$, and $G_{m'}(M_n(\mathbb{FZ}_2))$ satisfies Capelli identities. However, from \Cref{nocapelli}, if $\ell>0$ then both algebras cannot satisfy the same set of polynomial identities. Hence, $\ell=0$. If $m'>1$, then \cite[Proposition 6]{Frenkel} tells us that $\mathrm{Cap}_{k^2+1}\notin\mathrm{Id}(G_{m'}(M_n(\mathbb{FZ}_2)))$, a contradiction. So, $m'=1$.

Now, assume that $m<\infty$. Since $\dim G_m(M_{k,\ell})<\infty$ and from \Cref{nocapelli}, one obtains that $m'<\infty$. From \Cref{exp}, one has
$$
k^2+\ell^2=\mathrm{exp}(G_m(M_{k,\ell}))=\mathrm{exp}(G_{m'}(M_n(\mathbb{FZ}_2)))=n^2.
$$
Thus, if $\ell>0$, then $n>k$. So $$
(\mathrm{Cap}_{k^2+1})^{m+1}\in\mathrm{Id}(G_m(M_{k,\ell}))\setminus\mathrm{Id}(G_{m'}(M_n(\mathbb{FZ}_2))),
$$
a contradiction. Hence, $\ell=0$. Now, once again from \cite[Proposition 6]{Frenkel}, we see that $m'=1$.
\end{proof}

Combining all the lemmas, we obtain the following result:
\begin{Thm}\label{superpol_assoc}
Let $\mathcal{A}$ and $\mathcal{B}$ be finite-dimensional $\mathbb{Z}_2$-graded-simple associative algebras over an algebraically closed field of characteristic zero. Let $m$, $m'\in\mathbb{N}\cup\{\infty\}$. If
$$
\mathrm{Id}(G_m(\mathcal{A}))=\mathrm{Id}(G_{m'}(\mathcal{B})),
$$
then one of the following holds valid:
\begin{enumerate}
\item $\mathcal{A}\cong_2 M_n(\mathbb{F})$, $\mathcal{B}\cong_2 M_n(\mathbb{FZ}_2)$, and $m'=1$, or
\item the same as above, where $\mathcal{A}$ and $\mathcal{B}$, and $m$ and $m'$ switch places, or
\item $\mathcal{A}\cong_2\mathcal{B}$. In addition:
\begin{enumerate}
\item if $\mathcal{A}\cong_2M_n(\mathbb{FZ}_2)$, then $\lfloor\frac{m}2\rfloor=\lfloor\frac{m'}2\rfloor$.
\item if $\mathcal{A}\cong_2M_{k,\ell}$, then $m=m'$,
\end{enumerate}
\end{enumerate}
Conversely, if any of the conditions above holds valid, then $\mathrm{Id}(G_m(\mathcal{A}))=\mathrm{Id}(G_{m'}(\mathcal{B}))$.
\end{Thm}
\begin{proof}
We separate in some cases. If $\mathcal{A}\cong_2 M_n(\mathbb{FZ}_2)\cong_2\mathcal{B}$, then from \Cref{casemnfz2} we obtain the case (3)(a). If $\mathcal{A}\cong_2 M_{k,\ell}$ and $\mathcal{B}\cong_2 M_{k',\ell'}$, then we apply \Cref{casemkl} to be in the situation (3)(b). If $\mathcal{A}\cong_2 M_{k,\ell}$ and $\mathcal{B}\cong_2M_n(\mathbb{FZ}_2)$, then we apply \Cref{distinctalg} to obtain the assertion (1) (or (2) if $\mathcal{A}$ and $\mathcal{B}$ switch places).

The converse follows from the following. The situations (1) and (2) are described in \Cref{g1_ordinary}; the situation (3)(b) is immediate, while the situation (3)(a) is given by \Cref{casemnfz2}.
\end{proof}

\section{$\Omega$-algebras\label{omega_alg}}

Now, we shall provide similar constructions to non-necessarily associative algebras. Indeed, we shall prove the results in the context of $\Omega$-algebras (see the definition below). We write $\Omega=\cup_{n\ge0}\Omega_n$, and we always assume that $\Omega_n\ne\emptyset$ for some $n\ge2$. For convenience, we include some basic definitions.

\subsection{Preliminaries}
Let $\Omega=\bigcup_{m=0}^\infty\Omega_m$ be a set, called \textit{signature}. An \textit{$\Omega$-algebra} is a vector space $\mathcal{A}$ such that every $\omega\in\Omega_m$ defines an $m$-ary operation on $\mathcal{A}$, that is, a linear map $\omega:\underbrace{\mathcal{A}\otimes\cdots\otimes\mathcal{A}}_\text{$m$ times}\to\mathcal{A}$.

A vector subspace $I$ of $\mathcal{A}$ is called an \textit{ideal} if for any $\omega\in\Omega_n$, $n\ge1$, all $a_1,\ldots,a_{s-1},a_{s+1},\ldots,a_n\in\mathcal{A}$ (where $1\le s\le n$) and all $b\in I$, one has
\[
\omega(a_1,\ldots,a_{s-1},b,a_{s+1},\ldots,a_n)\in I.
\]

Let $\mathcal{B}$ be another $\Omega$-algebra. A linear map $f:\mathcal{A}\to\mathcal{B}$ is called a \textit{homomorphism} of $\Omega$-algebras if
\[
f(\omega(a_1,a_2,\ldots,a_n))=\omega(f(a_1),f(a_2),\ldots,f(a_n)),
\]
for all $a_1,a_2,\ldots,a_n\in\mathcal{A}$ and all $\omega\in\Omega$, where $\omega$ is an $n$-ary operation. If, moreover, $f$ is bijective, then we call $f$ an \textit{isomorphism} of $\Omega$-algebras. In this case, we say that $\mathcal{A}$ and $\mathcal{B}$ are isomorphic $\Omega$-algebras. If $\mathcal{A}=\mathcal{B}$ and $f$ is an isomorphism, then we call $f$ an \textit{automorphism}.

Given a  non-empty set $X$, one can define the free $\Omega$-algebra $\mathbb{F}_\Omega\langle X\rangle $ as follows.  First we build the set $W=W_\Omega(X)$ of  $\Omega$-monomials in $X$ as the union of subsets $W_n$, $n=0,1,2,\ldots$ given by $W_0=\Omega_0\cup X$ and for $n>0$,
\[
W_n=\bigcup_{m=1}^\infty\bigcup_{\omega\in\Omega_m}\bigcup_{i_1+\cdots+i_m+1\le n} \omega(W_{i_1},\ldots,W_{i_m}).
\] 
From this definition, it follows that for any $\omega\in \Omega_m$ and any $a_1,\ldots,a_m\in W$ the expression $\omega(a_1,\ldots,a_m)$ is a well-defined element of $W$. The elements of $W_n$ are called \textit{monomials} of degree $n$.

Then we consider the linear span $\mathbb{F}_\Omega\langle X\rangle $ of $W=W_\Omega(X)$. If $F_n=\text{Span}\{ W_n\}$ then $\mathbb{F}_\Omega\langle X\rangle =\bigoplus_{n=0}^\infty F_n$. The elements of $F_n$ are called \textit{homogeneous polynomials of degree $n$}. In a usual way, one defines the degree of an arbitrary nonzero polynomial. By linearity, every $\omega\in\Omega_m$ defines an $m$-linear operation on $\mathbb{F}_\Omega\langle X\rangle$. Also, it follows that for any $\Omega$-algebra $\mathcal{A}$ any map $\varphi:X\to\mathcal{A}$ uniquely extends to a homomorphism $\bar{\varphi}: \mathbb{F}_\Omega\langle X\rangle \to\mathcal{A}$. The $\Omega$-algebra $\mathbb{F}_\Omega\langle X\rangle $ is called the \textit{free $\Omega$-algebra} with the \textit{basis} (set of free generators) $X$.

The equation of the form $f(x_1,\ldots,x_n)=0$ where $f(x_1,\ldots,x_n)\in \mathbb{F}_\Omega\langle X\rangle $ is called a \textit{(polynomial) identity} in an $\Omega$-algebra $\mathcal{A}$ if under any map $\varphi: X\to\mathcal{A}$ one has $\bar{\varphi}(f(x_1,\ldots,x_n))=0$. In other way, $f(a_1,\ldots,a_n)=0$, for any $a_1,\ldots, a_n\in\mathcal{A}$. The set of all polynomial identities of $\mathcal{A}$ is denoted by $\mathrm{Id}_\Omega(\mathcal{A})$.


For a polynomial $g\in F_\Omega$, we use the notation $g\mid_{x_i=a}$ to denote the value of $g$ when $x_i$ is replaced by  $a$, and we use a  similar notation for the evaluation in more than one variables. Let $g=g(x_1,\ldots,x_m)\in \mathbb{F}_\Omega\langle X\rangle $. We say that $g$ is linear in the variable $x_i$ if
\[
g\mid_{x_i=x_i+\lambda y_i}=g+\lambda g\mid_{x_i=y_i},
\]
where $\lambda\in\mathbb{F}$. A multilinear polynomial is a polynomial which is linear in each of its variables.

Now, we shall assume that there exists $m>1$ such that $\Omega_m\ne\emptyset$. Let $\mathcal{A}$ be an $\Omega$-algebra and $S_1$, $S_2\subseteq\mathcal{A}$. We define the product of $S_1$ and $S_2$ as the subspace spanned by all $f(s_1,s_2,a_1,\ldots,a_r)$, where $s_1\in S_1$, $s_2\in S_2$, $a_1$, \dots, $a_r\in\mathcal{A}$, and $f=f(x_1,x_2,y_1,\ldots,y_r)\in\mathbb{F}_\Omega\langle X\rangle$ is linear in $x_1$ and $x_2$. We denote $S^2=S\cdot S$. Note that we get the same definition if we require the polynomial $f$ to be multilinear. Then, we say that:
\begin{enumerate}
\item $\mathcal{A}$ is a \emph{simple} $\Omega$-algebra if $\mathcal{A}^2\ne0$ and $\mathcal{A}$ has no proper nontrivial ideals.
\item $\mathcal{A}$ is \emph{prime} if for any pair of nonzero ideals $I_1$, $I_2\subseteq\mathcal{A}$, then $I_1\cdot I_2\ne0$,
\item $\mathcal{A}$ is \emph{semiprime} if for any nonzero ideal $I\subseteq\mathcal{A}$, one has $I^2\ne0$.
\end{enumerate}

\subsubsection{Gradings on $\Omega$-algebras} The notion of a group grading extends naturally to an $\Omega$-algebra (see, for instance, \cite{EK2013}). We present here only the situation where the grading group is $\mathbb{Z}_2$.

A $\mathbb{Z}_2$-grading on an $\Omega$-algebra $\mathcal{A}$ is a vector space decomposition
$$
\mathcal{A}=\mathcal{A}^0\oplus\mathcal{A}^1,
$$
where for each $\omega\in\Omega_n$, one has
$$
\omega(\mathcal{A}^{i_1},\ldots,\mathcal{A}^{i_n})\subseteq\mathcal{A}^{i_1+\cdots+i_n}.
$$
A graded homomorphism (isomorphism) between two $\mathbb{Z}_2$-graded $\Omega$-algebras is a homomorphism (isomorphism) of $\Omega$-algebras which is also a graded linear map. If there exists an isomorphism between two $\Omega$-algebras, say $\mathcal{A}$ and $\mathcal{B}$, then we say that they are graded-isomorphic and denote $\mathcal{A}\cong_2\mathcal{B}$.

The free $\mathbb{Z}_2$-graded $\Omega$-algebra is the free $\Omega$-algebra, freely generated by $X^{\mathbb{Z}_2}=X^0\cup X^1$, where $X^i=\{x_1^{(i)},x_2^{(i)},\ldots\}$. The free $\Omega$-algebra $\mathbb{F}_\Omega\langle X^{\mathbb{Z}_2}\rangle$ has a natural $\mathbb{Z}_2$-grading, and it satisfies the universal property for the $\mathbb{Z}_2$-graded $\Omega$-algebras. We define a $\mathbb{Z}_2$-graded polynomial identity in the standard way. The set of all $\mathbb{Z}_2$-graded polynomial identities of $\mathcal{A}$ is denoted by $\mathrm{Id}_2(\mathcal{A})$. We shall denote $Y=X^0$ and $Z=X^1$, and denote $y_k=x_k^{(0)}$, $z_k=x_k^{(1)}$, for each $k\in\mathbb{N}$.

\subsubsection{Codimension sequence} It will be relevant for us to consider the codimension sequence of an $\Omega$-algebra. We denote by $P_{m,\Omega}$ the set of all multilinear polynomials of $\mathbb{F}_\Omega\langle X\rangle$ in the variables $x_1$, \dots, $x_m$. Note that one may have $\dim P_{m,\Omega}=\infty$. Given an $\Omega$-algebra $\mathcal{A}$, we set
$$
c_m(\mathcal{A})=\dim P_{m,\Omega}/P_{m,\Omega}\cap\mathrm{Id}_\Omega(\mathcal{A}),\quad m\in\mathbb{N}.
$$
In a similar manner as in the ordinary case, we define a variety $\mathscr{V}$ of $\Omega$-algebras as the class of all $\Omega$-algebras satisfying a given set of polynomials identities. The set of polynomials identities satisfied by all the algebras in a given class $\mathscr{V}$ is denoted by $\mathrm{Id}_\Omega(\mathscr{V})$. Then, we set
$$
c_m(\mathscr{V})=\dim P_{m,\Omega}/P_{m,\Omega}\cap\mathrm{Id}_\Omega(\mathscr{V}),\quad m\in\mathbb{N}.
$$

\subsection{Grassmann envelope\label{grassenv}}
Let $\mathcal{A}$ be an $\Omega$-algebra. The tensor product $\mathcal{A}\otimes G$ has a structure of $\Omega$-algebra, where
$$
\mu\mapsto\mu\otimes1\in\mathcal{A}\otimes G,\quad\mu\in\Omega_0,
$$
and, for any $\omega\in\Omega_n$ ($n>0$), one has
$$
\omega:(a_1\otimes g_1,\ldots,a_n\otimes g_n)\in(\mathcal{A}\otimes G)^n\mapsto\omega(a_1,\ldots,a_n)\otimes g_1\ldots g_n\in\mathcal{A}\otimes G.
$$
Now, if $\mathcal{A}=\mathcal{A}^0\oplus\mathcal{A}^1$ is a $\mathbb{Z}_2$-graded $\Omega$-algebra, then we set
$$
G(\mathcal{A})=\mathcal{A}^0\otimes G^0\oplus\mathcal{A}^1\otimes G^1\subseteq\mathcal{A}\otimes G.
$$
Thus, $G(\mathcal{A})$ becomes an $\Omega$-algebra. When $\Omega=\Omega_2=\{\cdot\}$ (that is, $\mathcal{A}$ is a binary algebra), $G(\mathcal{A})$ coincides with the usual Grassmann envelope of $\mathcal{A}$. Note that $G(\mathcal{A})$ has a natural $\mathbb{Z}_2$-grading, where $(G(\mathcal{A}))^0=\mathcal{A}^0\otimes G^0$, and $(G(\mathcal{A}))^1=\mathcal{A}^1\otimes G^1$.

With these operations, $G(\mathcal{A})$ (and $\mathcal{A}\otimes G$) becomes a $G^0$-algebra. It means that $G(\mathcal{A})$ is a $G^0$-module, and every operation on $G(\mathcal{A})$ is $G^0$-linear.

The first results are standard and analogous to the ordinary case. Recall that every word in the free algebra $\mathbb{F}_\Omega\langle X\rangle$ is uniquely determined if we suppress all the parenthesis of the operations \cite[Lemma 2.7]{JacBAII}.

\begin{Def}[$\ast$-operation]
Let $f=f(y_1,\ldots,y_r,z_1,\ldots,z_s)\in\mathbb{F}_\Omega\langle X^{\mathbb{Z}_2}\rangle$ be a multilinear polynomial, and write
$$
f=\sum_w\sum_{\sigma\in\mathcal{S}_r}\lambda_{w,\sigma} w_0z_{\sigma(1)}w_1\cdots w_{s-1} z_{\sigma(s)}w_s,
$$
where each $w_i$ can be empty, and contains variables of trivial degree and operations; and the sum varies between distinct sequences $w=(w_0,w_1,\ldots,w_s)$. We define
$$
f^\ast=\sum_w\sum_{\sigma\in\mathcal{S}_r}(-1)^\sigma\lambda_{w,\sigma}w_0z_{\sigma(1)}w_1\cdots w_{s-1}z_{\sigma(s)}w_s,
$$
\end{Def}

From a standard computation, we have:
\begin{Lemma}\label{supergrid}
Let $f=f(y_1,\ldots,y_r,z_1,\ldots,z_s)\in\mathbb{F}_\Omega\langle X^{\mathbb{Z}_2}\rangle$ be multilinear in all of its variables. Then, for any admissible substitution on $G(\mathcal{A})$, we have
$$
f(a_1\otimes h_1,\ldots, a_r\otimes h_r,b_1\otimes g_1,\ldots, b_s\otimes g_s)=f^\ast(a_1,\ldots,a_r,b_1,\ldots,b_s)\otimes g,
$$
where $g=h_1\cdots h_rg_1\cdots g_s$.\qed
\end{Lemma}
As one of the immediate consequences, we have:
\begin{Cor}
Let $f\in\mathbb{F}_\Omega\langle X^{\mathbb{Z}_2}\rangle$ be a multilinear polynomial, and let $\mathcal{A}$ be a $Z_2$-graded $\Omega$-algebra. Then $f\in\mathrm{Id}_2(\mathcal{A})$ if and only if $f^\ast\in\mathrm{Id}_2(G(\mathcal{A}))$.\qed
\end{Cor}

In particular, we have the following consequence. Let $\mathcal{A}$ and $\mathcal{B}$ be $\mathbb{Z}_2$-graded $\Omega$-algebras. If $\mathrm{char}\,\mathbb{F}=0$, then $\mathrm{Id}_2(\mathcal{A})=\mathrm{Id}_2(\mathcal{B})$ if and only if $\mathrm{Id}_2(G(\mathcal{A}))=\mathrm{Id}_2(G(\mathcal{B}))$. If $\mathcal{A}$ and $\mathcal{B}$ are prime and finite-dimensional over $\mathbb{F}$, and $\mathbb{F}$ is algebraically closed, then any of the former equalities implies $\mathcal{A}\cong_2\mathcal{B}$ (\cite{BY19a}). If two $\mathbb{Z}_2$-graded algebras satisfy the same set of superpolynomial identities, i.e., $\mathrm{Id}(G(\mathcal{A}))=\mathrm{Id}(G(\mathcal{B}))$, then it does not readily imply that the set of $\mathbb{Z}_2$-graded polynomial identities are the same. Thus, it is not trivial to conclude that $\mathcal{A}\cong_2\mathcal{B}$.

Using the $\ast$-operation, we can define supervarieties of $\Omega$-algebras:
\begin{Def}
Let $\mathscr{V}$ be a variety of $\mathbb{Z}_2$-graded $\Omega$-algebras. Then, the variety of $\mathscr{V}$-super $\Omega$-algebras is the class $\mathscr{V}^\ast$ of all $\mathbb{Z}_2$-graded $\Omega$-algebras $\mathcal{A}$ such that $G(\mathcal{A})\in\mathscr{V}$.
\end{Def}
Another consequence of \Cref{supergrid} is as follows. A $\mathscr{V}$-supervariety is a variety of $\mathbb{Z}_2$-graded $\Omega$-algebras. Moreover:
\begin{Cor}
Let $\{f_i\mid i\in\mathscr{I}\}$ be a set of multilinear polynomials in $\mathbb{F}_\Omega\langle X^{\mathbb{Z}_2}\rangle$ defining a variety $\mathscr{V}$ of $\mathbb{Z}_2$-graded $\Omega$-algebras. Then, $\mathscr{V}^\ast$ is a variety of $\mathbb{Z}_2$-graded $\Omega$-algebras defined by $\{f_i^\ast\mid i\in\mathscr{I}\}$.\qed
\end{Cor}

Now, we have seen that $\mathcal{A}$ satisfies a $\mathbb{Z}_2$-graded PI if and only if $G(\mathcal{A})$ satisfies as well. Next, we investigate the same question in the context of ordinary $\Omega$-polynomial identities. It is clear that if $G(\mathcal{A})$ satisfies a PI, then $\mathcal{A}$ satisfies the same PI (since $\mathcal{A}\subseteq G(\mathcal{A})$). Since we assume that $\mathrm{char}\,\mathbb{F}=0$, we need only to deal with multilinear polynomial identities. As before, we denote by $G_m$ the finite-dimensional Grassmann algebra generated by $m$ elements, and $G_m(\mathcal{A})$ denotes the respective Grassmann envelope.

\begin{Lemma}
Let $m\in\mathbb{N}$ and $f=f(x_1,\ldots,x_m)\in\mathbb{F}_\Omega\langle X\rangle$ be multilinear. Then, $f\in\mathrm{Id}_\Omega(G(\mathcal{A}))$ if and only if $f\in\mathrm{Id}_\Omega(G_m(\mathcal{A}))$. Moreover, $f\in\mathrm{Id}_\Omega(\mathcal{A}\otimes G)$ if and only if $f\in\mathrm{Id}_\Omega(\mathcal{A}\otimes G_m)$.
\end{Lemma}
\begin{proof}
Since $G_m(\mathcal{A})\subseteq G(\mathcal{A})$, it is clear that $\mathrm{Id}_\Omega(G(\mathcal{A}))\subseteq\mathrm{Id}_\Omega(G_m(\mathcal{A}))$. Now, assume that $f\in\mathrm{Id}_\Omega(G_m(\mathcal{A}))$ is multilinear. Since $f$ is multilinear, it is enough to check that $f$ annihilates a set of generators of $G(\mathcal{A})$. So, let $a_1\otimes g_1$, \dots, $a_m\otimes g_m\in G(\mathcal{A})$. Write $g_i=h_id_i$, where $h_i\in G^0$ and either $d_i=1$ or $d_i=e_{\ell_i}$. Then,
$$
f(a_1\otimes g_1,\ldots,a_m\otimes g_m)=f(a_1\otimes d_1,\ldots,a_m\otimes d_m)h,
$$
where $h=h_1\cdots h_m\in G_m$. Up to renaming generators, one has $a_1\otimes d_1$, \dots, $a_m\otimes d_m\in G_m(\mathcal{A})$. Thus, $f(a_1\otimes d_1,\ldots,a_m\otimes d_m)=0$, and $f\in\mathrm{Id}_\Omega(G(\mathcal{A}))$. A similar argument holds valid for $\mathcal{A}\otimes G_m$ and $\mathcal{A}\otimes G$.
\end{proof}

\begin{Thm}
Let $\mathcal{A}$ be a $\mathbb{Z}_2$-graded $\Omega$-algebra such that $c_m(\mathcal{A})<\infty$, for all $m\in\mathbb{N}$. Then
$$
c_m(\mathcal{A})\le c_m(G(\mathcal{A}))\le 2^{m}c_m(\mathcal{A}).
$$
\end{Thm}
\begin{proof}
Let $m\in\mathbb{N}$. From the previous lemma, it is enough to compute $P_{m,\Omega}(G_m(\mathcal{A}))$. Every multilinear polynomial $f\in P_{m,\Omega}(G_m(\mathcal{A}))$ can be viewed as a multilinear map
$$
(a_1,\ldots,a_m,g_1,\ldots,g_m)\in\mathcal{A}^m\times G_m^m\mapsto f(a_1\otimes g_1,\ldots,a_m\otimes g_m)\in\mathcal{A}\otimes G_m.
$$
This identification is clearly injective. For, the existence of a multilinear polynomial having 0 as its image implies that it is a polynomial identity for $G_m(\mathcal{A})$. Now, from \Cref{supergrid}, the former equation equals $f^\ast(a_1,\ldots,a_m)\otimes g_1\cdots g_m$. Thus, $\dim P_{m,\Omega}(G_m(\mathcal{A}))$ is bounded by $\dim(P_{m,\Omega}(\mathcal{A})\otimes G_m)$. The first one has dimension $c_m(\mathcal{A})$, while the second has dimension $2^{m}$. Hence,
$$
c_m(G(\mathcal{A}))=c_m(G_m(\mathcal{A}))\le 2^{m}c_m(\mathcal{A}),
$$
for all $m\in\mathbb{N}$.
\end{proof}

As a consequence, we have the following:
\begin{Cor}
Let $\mathcal{A}$ be a $\mathbb{Z}_2$-graded $\Omega$-algebras, and $\mathscr{V}$ be a variety of $\Omega$-algebras such that $\dim P_{m,\Omega}(\mathscr{V})$ asymptotically grows strictly faster than $2^{m}c_m(\mathscr{W})$ for any proper subvariety $\mathscr{W}\subset\mathscr{V}$, and assume that $\mathcal{A}$, $G(\mathcal{A})\in\mathscr{V}$. Denote by $\mathbb{F}(\mathscr{V})$ its relatively free algebra. Let $\mathcal{A}$ be a $\mathbb{Z}_2$-graded $\Omega$-algebra. Then $\mathcal{A}$ satisfies a PI from $\mathbb{F}(\mathscr{V})$ if and only if $G(\mathcal{A})$ satisfies as well.
\end{Cor}
\begin{proof}
Since $\mathcal{A}\subseteq G(\mathcal{A})$, every polynomial identity satisfied by $G(\mathcal{A})$ is satisfied by $\mathcal{A}$ as well. Conversely, assume that $\mathcal{A}$ satisfies a polynomial identity in $\mathbb{F}(\mathscr{V})$. It means that $2^{m}c_m(\mathcal{A})<c_m(\mathscr{V})$, for some $m\in\mathbb{N}$. From the previous theorem, one has
$$
c_m(G(\mathcal{A}))\le2^{m}c_m(\mathcal{A})<c_m(\mathscr{V}).
$$
Thus, $G(\mathcal{A})$ satisfies a polynomial identity in $\mathbb{F}(\mathscr{V})$.
\end{proof}

The former result is slightly stronger than the statement ``$\mathcal{A}$ is a PI-algebra if and only if $G(\mathcal{A})$ is a PI-algebra". Consider, for instance, the associative case. Then $\mathcal{A}$ is associative if and only if $G(\mathcal{A})$ is associative. Thus, both algebras satisfy the polynomial identity given by the associator $(xy)z-x(yz)$. However, it is known that we have a stronger result: if $\mathcal{A}$ satisfies a polynomial identity as an associative algebra, then $G(\mathcal{A})$ also satisfies a polynomial identity from the free associative algebra. The associative context is a particular case of the statement of the previous corollary. Indeed, let $\mathscr{V}$ be the variety of all associative algebras. The variety $\mathscr{V}$ satisfies $c_m(\mathscr{V})=m!$, and Regev's result states that every proper subvariety of $\mathscr{V}$ is exponentially bounded (see \cite{Regev}). Thus, $2^mc_m(\mathscr{W})<c_m(\mathscr{V})$, for convenient values of $m$ and any proper subvariety $\mathscr{W}\subset\mathscr{V}$. Hence, $\mathscr{V}$ satisfies the hypothesis of the previous result.

\begin{Remark}
As a final remark, we have that $\mathrm{Id}_\Omega(\mathcal{A})=\mathrm{Id}_\Omega(G(G(\mathcal{A})))$. Indeed, we shall prove that both $\Omega$-algebras satisfy the same set of $\mathbb{Z}_2$-graded multilinear polynomial identities. Let $f=f(y_1,\ldots,y_r,z_1,\ldots,z_r)\in\mathbb{F}_\Omega\langle X^{\mathbb{Z}_2}\rangle$ such that $f\in\mathrm{Id}_2(\mathcal{A})$. It is clear that $f^{\ast\ast}=f$. Then, for any homogeneous $(a_i\otimes g_i)\otimes h_i$, $(a_j'\otimes g_j')\otimes h_j'\in G(\mathcal{A})$, we have, from \Cref{supergrid},
\begin{align*}
f((a_1\otimes g_1)\otimes h_1,\ldots,(a_s'\otimes g_s')\otimes h_s')&=f^\ast(a_1\otimes g_1,\ldots,a_s'\otimes g_s')\otimes h\\&=f(a_1,\ldots,a_s')\otimes g\otimes h=0,
\end{align*}
where $g=g_1\ldots g_r g_1'\cdots g_s'$ and $h=h_1\cdots h_r h_1'\cdots h_s'$. Thus, $f\in\mathrm{Id}_2(G(G(\mathcal{A})))$. Since both algebras satisfy the same set of multilinear $\mathbb{Z}_2$-graded polynomial identities, then they satisfy the same set of $\mathbb{Z}_2$-graded polynomial identities ($\mathrm{char}\,\mathbb{F}=0$). Moreover, as a consequence, one has $\mathrm{Id}_\Omega(\mathcal{A})=\mathrm{Id}_\Omega(G(G(\mathcal{A})))$.
\end{Remark}

\subsection{Polynomial identities of the Grassmann algebra}
In this subsection, we shall compute the $\Omega$-polynomial identities of the Grassmann algebra. The result obtained here will be important to obtain a generalization of \Cref{casemnfz2} in the context of $\Omega$-algebras. The natural structure of $\Omega$-algebra on $G$ (and on $G_m$) is as follows. First, for any $\mu\in\Omega_0$, we shall associate
$$
\mu\mapsto1\in G.
$$
If $\omega\in\Omega_n$, then, we set
$$
\omega(g_1,\ldots,g_n)=g_1\cdots g_n,\quad g_1,\ldots,g_n\in G.
$$
Before we proceed, it is worth mentioning the following. If $\mathcal{A}$ and $\mathcal{B}$ are $\Omega$-algebras, then $\mathcal{A}\otimes\mathcal{B}$ is an $\Omega$-algebra as well, via the map satisfying:
$$
\omega:(a_1\otimes b_1,\ldots, a_n\otimes b_n)\in(\mathcal{A}\otimes\mathcal{B})^n\mapsto\omega(a_1,\ldots,a_n)\otimes\omega(b_1,\ldots,b_n)\in\mathcal{A}\otimes\mathcal{B}.
$$
Then, given any $\mathbb{Z}_2$-graded $\Omega$-algebra $\mathcal{A}$, the above defined $\Omega$-operations on $G$ gives operations on the tensor product $\mathcal{A}\otimes G$. These operations agree with the natural one defined in \Cref{grassenv}.

First, note the following polynomial identities satisfied by $G$:
\begin{enumerate}
\item $\omega(x)-x\in\mathrm{Id}_\Omega(x)$, for any $\omega\in\Omega_1$. Thus, operations in $\Omega_1$ may be ignored.
\item given $\omega\in\Omega_n$, we have several associativity relations that are polynomial identities for $G$. For instance, if $n=3$, then one of the many possible combinations is:
$$
\omega(x_1,\omega(x_2,x_3,x_4),x_5)-\omega(\omega(x_1,x_2,x_3),x_4,x_5)\in\mathrm{Id}_\Omega(G).
$$
\item if $\omega$, $\omega'\in\Omega_n$, then $\omega-\omega'\in\mathrm{Id}_\Omega(G)$. Thus, we may fix one $\omega\in\Omega_n$, and any other $\omega'\in\Omega_n$ can be represented via $\omega$, modulo $\mathrm{Id}_\Omega(G)$.
\item if $\omega\in\Omega_n$ and $\omega'\in\Omega_{t}$, where $t>n$, then
$$
\omega'(x_1,\ldots,x_t)-\underbrace{\omega\cdots\omega}_\text{$t-n+1$ times}(\cdots(x_1,\ldots,x_n),\ldots,x_t)\in\mathrm{Id}_\Omega(G).
$$
It means that we may choose a $\omega\in\Omega_n$, where $n>1$ is the first index such that $\Omega_n\ne\emptyset$. Any other operation may be represented via $\omega$, modulo $\mathrm{Id}_\Omega(G)$.
\item If $\Omega_0\ne\emptyset$ and $\mu\in\Omega_0$, then for any $\omega\in\Omega_n$ (where $n>1$), we have a bilinear polynomial
$$
p(x,y)=\omega(x,y,\mu,\ldots,\mu)\in\mathbb{F}_\Omega\langle X\rangle.
$$
This polynomial behaves like a binary operation on $G$, and, in particular, the polynomial identities (2)-(4) hold valid if we replace $\omega$ by $p$.
\end{enumerate}
Let $T$ be the $T_\Omega$-ideal generated by all of the above polynomial identities. The identities (1)-(5) mean that we may replace the free $\Omega$-algebra by the relatively free $\Omega$-algebra $\mathbb{F}_\Omega\langle X\rangle/T$. If $\Omega_0=\emptyset$, then, from identities (3) and (4), this relatively free algebra can be seen as the set of all words of length at least $n$, where $n>1$ is the first index such that $\Omega_n\ne\emptyset$. From identity (2), these words can be left-normed. In other words, if we suppress the operations, then we can identify the words of $\mathbb{F}_\Omega\langle X\rangle/T$ with the words in the free associative algebra. In the special case where $\Omega_0\ne\emptyset$, then $\mathbb{F}_\Omega\langle X\rangle/T$ coincides with the free associative algebra $\mathbb{F}\langle X\rangle$. Hence, a basis of $\mathbb{F}_\Omega\langle X\rangle/T$ is described below:
\begin{enumerate}
\renewcommand{\labelenumi}{(\roman{enumi})}
\item If $\Omega_0\ne\emptyset$, then a basis consists of $1$ and $x_{i_1}\cdots x_{i_m}$, for $m\ge1$, and $x_{i_1}$, \dots, $x_{i_m}\in X$.
\item If $\Omega_0=\emptyset$ and $n>1$ is the first index such that $\Omega_n\ne\emptyset$, then a basis consists of all words $x_{i_1}\cdots x_{i_m}$, where $m\ge n$, and $x_{i_1}$, \dots, $x_{i_m}\in X$.
\end{enumerate}

In particular, we can identify the set of multilinear polynomials $P_{m,\Omega}/T\cap P_{m,\Omega}$ and the multilinear associative polynomials $P_m$, for all $m\ge n$. Moreover, we have $P_{m,\Omega}(G)=P_m(G)$. Thus, it is enough to find a set of polynomial identities in $I\subseteq\mathrm{Id}_\Omega(G)$ such that $P_m(G)\equiv_IP_{m,\Omega}$. If either $\Omega_2\ne\emptyset$, or $\Omega_0\ne\emptyset$, then the description of the polynomial identities $\mathrm{Id}_\Omega(G)$ coincides with the classical associative case (up to all the polynomial identities (1)-(5)).

\begin{Remark}\label{examp}
Let $\mathcal{A}$ be a (binary) associative algebra and assume that $\Omega_0=\emptyset$. Then, for any $\omega\in\Omega_n$ ($n>0$), we can define
$$
\omega:(a_1,\ldots,a_n)\in\mathcal{A}^n\mapsto a_1\cdots a_n\in\mathcal{A}.
$$
Using this, the polynomial identities (1)-(5) holds valid. Thus, $\mathcal{A}$ belongs to the variety of $\Omega$-algebras generated by $T$.
\end{Remark}

\begin{Thm}
Consider the relatively free algebra $\mathbb{F}_\Omega\langle X\rangle/T$. Assume that $\Omega_0\ne\emptyset$. Then,
$$
\mathrm{Id}_\Omega(G)=\langle[x_1,x_2,x_3]\rangle^T.
$$
If $\Omega_0=\emptyset$ and $n>1$ is the least integer such that $\Omega_n\ne\emptyset$, then a set of generators of $\Omega$-polynomial identities of $G$ is
$$
\left\{x_1\cdots x_i[x_a,x_b]x_{i+1}\cdots x_{n-2}-x_1\cdots x_{n-2}[x_1,x_b]\mid i=0,1,\ldots,n-3\right\}.
$$
\end{Thm}
\begin{proof}
First, recall that, as an associative algebra, $\mathrm{Id}(G)=\langle[x_1,x_2,x_3]\rangle^T$, and a basis of $P_n(G)$ consists of all $x_{i_1}\cdots x_{i_r}[x_{j_1},x_{j_2}]\ldots[x_{j_{s-1}},x_{j_s}]$, where $s,r\ge0$, $s+r=n$, $s$ is even, $i_1<\cdots<i_r$, and $j_1<\cdots<j_s$ (see, for instance, \cite{Drenskybook}).

Now, let us assume that $\Omega_0=\emptyset$ and $n>1$ is the first index such that $\Omega_n\ne\emptyset$. With all the notations set, we shall compute the polynomial identities of the Grassmann algebra. First, note that, for any $i=0,1,\ldots,n-3$,
\begin{equation}\label{idG}
x_1\cdots x_i[x_a,x_b]x_{i+1}\cdots x_{n-2}-x_1\cdots x_{n-2}[x_1,x_b]\in\mathrm{Id}_\Omega(G).
\end{equation}
In particular, when $i=n-3$, the described polynomial is the triple commutator up to a sequence of variables, that is, $x_1\cdots x_{n-3}[x_{n-2},x_{n-1},x_n]$. Moreover, the difference between two consecutive identities of \eqref{idG} gives the triple commutator $x_1\cdots x_i[x_a,x_b,x_c]x_{i+1}\cdots x_{n-3}$ in the middle of a sequence of variables.

Now, we shall prove that the polynomials
$$
x_1\cdots x_i[x_a,x_b][x_c,x_d]x_{i+1}\cdots x_{n-4}-x_1\cdots x_i[x_a,x_c][x_b,x_d]x_{i+1}\cdots x_{n-4},
$$
for all $i=0,1,\ldots,n-4$, are consequence of \eqref{idG}. This follows from the fact that $x_1\cdots x_i[x_a,x][x,x_d]x_{i+1}\cdots x_{n-4}$ is a consequence of the same set of identities, which can be checked using standard computations.

It is known that the space of multilinear associative polynomials may be written as a linear combination of
$$
x_{i_1}\cdots x_{i_r}c_1\cdots c_t,
$$
where $i_1<\cdots<i_r$ and $c_1$, \dots, $c_t$ are commutators (see, for instance, \cite{Drenskybook}). Then, it is clear that the set
$$
x_{i_1}\cdots x_{i_r}[x_{j_1},x_{j_2}]\cdots [x_{j_{s-1}},x_{j_s}],\quad r,s\ge0,\,\text{$s$ even},\,i_1<\cdots<i_r,\,j_1<\cdots<j_s,
$$
generates $P_m$, modulo the T-ideal generated by $\eqref{idG}$. From the previous discussion, we obtain that this is a set of generators for $\mathrm{Id}_\Omega(G)$. 
\end{proof}

\begin{Problem}
Find a minimal basis of polynomial identities of $\mathrm{Id}_\Omega(G)$.
\end{Problem}
From all the discussion and the proof of the previous theorem, the polynomial identities of the finite-dimensional Grassmann algebra is as follows.
\begin{Thm}\label{fdgrassid}
Assume that $\Omega_0\ne\emptyset$, and let $m\in\mathbb{N}$, $k=\lfloor\frac{m}2\rfloor$. Then
$$
\mathrm{Id}_\Omega(G_m)=\langle[x_1,x_2,x_3],[x_1,x_2]\cdots[x_{2k+1},x_{2k+2}]\rangle^T.
$$
If $\Omega_0=\emptyset$, and $n>1$ is the least integer such that $\Omega_n\ne\emptyset$, then
$$
\mathrm{Id}_{\Omega}(G_m)=\mathrm{Id}_{\Omega}(G)+\langle[x_1,x_2]\ldots[x_{2k+1},x_{2k+2}]x_{2k+3}\cdots x_n\rangle^T,
$$
where $x_{2k+3}\cdots x_n$ is empty if $n\le 2k+2$.
\end{Thm}
\begin{proof}
The finite dimensional Grassmann algebra $G_m$ has the set of associative polynomial identities given by $$
\mathrm{Id}(G_m)=\langle[x_1,x_2,x_3],[x_1,x_2]\cdots[x_{2k+1},x_{2k+2}]\rangle^T,
$$
and a basis of $P_n(G_m)$ is the same as of $G$, except that $s\le k$ (in the notation of the proof of the previous theorem). Now, the proof is a continuation of the arguments of the previous theorem. 
\end{proof}

In particular, $G_m$ and $G_{m'}$ satisfy the same set of $\Omega$-polynomial identities if and only if $\lfloor\frac{m}2\rfloor=\lfloor\frac{m'}2\rfloor$.

\subsection{Graded-simple algebras} First, we denote the vector space $\mathbb{FZ}_2=1\mathbb{F}\oplus\varepsilon\mathbb{F}$, and we define
$$
\varepsilon^m=\left\{\begin{array}{cc}1,&\text{ if $m$ is even},\\\varepsilon,&\text{ if $m$ is odd.}\end{array}\right.
$$
Using this terminology, we shall construct the algebra $\mathcal{A}\otimes\mathbb{FZ}_2$, where $\mathcal{A}$ is an arbitrary $\Omega$-algebra. As vector space, the elements of $\mathcal{A}\otimes\mathbb{FZ}_2$ are linear combination of elements of the kind $a:=a\otimes1$ and $\varepsilon a:=a\otimes\varepsilon$, for $a\in\mathcal{A}$. We define a structure of $\Omega$-algebra on $\mathcal{A}\otimes\mathbb{FZ}_2$ as follows. Given $\omega\in\Omega_n$ and $\varepsilon^{i_1}a_1$, \dots, $\varepsilon^{i_n}\in\mathcal{A}\otimes\mathbb{FZ}_2$, we set
$$
\omega(\varepsilon^{i_1}a_1,\ldots,\varepsilon^{i_n}a_n)=\varepsilon^{i_1+\cdots+i_n}\omega(a_1,\ldots,a_n).
$$
The given structure turns $\mathcal{A}\otimes\mathbb{FZ}_2$ into a $\mathbb{Z}_2$-graded $\Omega$-algebra. Note that, if we give the natural structure of $\Omega$-algebra on $\mathbb{FZ}_2$, as described in \Cref{examp}, then the just defined structure of $\Omega$-algebra on $\mathcal{A}\otimes\mathbb{FZ}_2$ coincides with the structure of $\Omega$-algebra on the tensor product of $\mathcal{A}$ and $\mathbb{FZ}_2$. Now, we have:

\begin{Lemma}\label{lem0}
Let $\mathcal{A}$ be an $\Omega$-algebra. Then $\mathcal{A}\otimes\mathbb{FZ}_2$ is graded-simple if and only if $\mathcal{A}$ is simple.
\end{Lemma}
\begin{proof}
Assume that $\mathcal{A}$ is simple as an $\Omega$-algebra. Then $(\mathcal{A}\otimes\mathbb{FZ}_2)^2\supseteq\mathcal{A}^2\ne0$. Let $I\subseteq\mathcal{A}\otimes\mathbb{FZ}_2$ be a $\mathbb{Z}_2$-graded ideal, and assume that $I\ne0$. Then, we can find $0\ne a+\varepsilon b\in I$, where $a$, $b\in\mathcal{A}$. Then, either $a\ne0$ or $b\ne0$ (or both are nonzero), so assume that $a\ne0$. Since $I$ is a graded ideal, it means that $a\in I$. Since $\mathcal{A}$ is simple, one has $a\cdot\mathcal{A}=\mathcal{A}$. Thus, for any $c\in\mathcal{A}$, we can find a multilinear polynomial $f=f(x_1,x_2,\ldots,x_m)\in\mathbb{F}_\Omega\langle X\rangle$ and $a_2$, \dots, $a_m\in\mathcal{A}$, such that $c=f(a,a_2,\ldots,a_m)$. It also implies that $\varepsilon c=f(\varepsilon a,a_2,\ldots,a_m)$. Thus, $c$, $\varepsilon c\in I$, so $I=\mathcal{A}\otimes\mathbb{FZ}_2$. An analogous argument holds if $b\ne0$.

On the other hand, assume that $\mathcal{A}$ is not simple. If $\mathcal{A}^2=0$, then the construction of $\mathcal{A}\otimes\mathbb{FZ}_2$ implies that $(\mathcal{A}\otimes\mathbb{FZ}_2)^2=0$. On the other hand, if $I\subseteq\mathcal{A}$ is an ideal, $I\ne0$ and $I\ne\mathcal{A}$, then it is elementary to see that $I\otimes\mathbb{FZ}_2$ is a graded ideal of $\mathcal{A}\otimes\mathbb{FZ}_2$. Hence, in any case, $\mathcal{A}\otimes\mathbb{FZ}_2$ is not graded-simple.
\end{proof}

Now, we shall prove the following description of $\mathbb{Z}_2$-graded-simple $\Omega$-algebras.
\begin{Thm}\label{gradedsimplealg}
Let $\mathcal{A}$ be a $\mathbb{Z}_2$-graded $\Omega$-algebra which is graded-simple, over a field $\mathbb{F}$ of characteristic not $2$. Then either $\mathcal{A}$ is simple as an $\Omega$-algebra, or $\mathcal{A}\cong_{2}\mathcal{B}\otimes\mathbb{FZ}_2$ as $\Omega'$-algebras, where $\mathcal{B}$ is a simple $\Omega'$-algebra, and $\Omega'=(\Omega_0\cap\mathcal{A}^0)\cup\bigcup_{n>0}\Omega_n$.
\end{Thm}
It is interesting to mention that, when $\Omega=\Omega_2=\{\cdot\}$ (that is, $\mathcal{A}$ is a binary algebra), this result is already known. Moreover, Elduque proves a more general result for an arbitrary grading group $G$ \cite{Elduque}.

We shall break the proof of \Cref{gradedsimplealg} in a few steps. From now on, on this subsection, unless otherwise explicitly mentioned, we shall assume that $\mathcal{A}=\mathcal{A}^0\oplus\mathcal{A}^1$ is a $\mathbb{Z}_2$-graded $\Omega$-algebra which is graded-simple. Moreover, we shall assume that $\mathcal{A}$ is not simple as an $\Omega$-algebra. We denote by $\pi_i:\mathcal{A}\to\mathcal{A}$ the projection $\mathcal{A}\to\mathcal{A}_i$ (with respect to the decomposition $\mathcal{A}=\mathcal{A}^0\oplus\mathcal{A}^1$), followed by the embedding $\mathcal{A}^i\hookrightarrow\mathcal{A}$. We shall use the same symbol to denote the projection $\pi_i:\mathcal{A}\to\mathcal{A}^i$.

\begin{Lemma}\label{lem1}
Let $I\subseteq\mathcal{A}$ be an ungraded ideal such that $I\ne0$ and $I\ne\mathcal{A}$. Then, the maps $\pi_i|_I:I\to\mathcal{A}_i$ are linear isomorphisms.
\end{Lemma}
\begin{proof}
Note that $I$ cannot contain a homogeneous element. Indeed, assume that $a\in I$ is homogeneous, and let $J$ be the graded-ideal of $\mathcal{A}$ generated by $a$. Then, $J$ is a nonzero homogeneous ideal contained in $I$. Since $\mathcal{A}$ is graded simple, then $\mathcal{A}=J\subseteq I$, a contradiction. Now, given $x\in\mathcal{A}$, the condition $\pi_i(x)=0$ implies either that $x$ is homogeneous of $\mathbb{Z}_2$-degree equals to $i+1$, or $x=0$. From the previous discussion the former is impossible, so we have that $\pi_i|_I$ is an one-to-one linear map. Now, let $x\in I\setminus\{0\}$, and let $y\in\mathcal{A}^0$. Then $\pi_0(x)\ne0$ and $\mathcal{A}=\mathcal{A}\cdot\pi_0(x)$. Thus, we can find a multilinear polynomial $f(x_1,x_2,\ldots,x_m)\in\mathbb{F}_\Omega\langle X\rangle$ and homogeneous $a_2$, $a_3$, \dots, $a_m\in\mathcal{A}$ such that $y=f(\pi_0(x),a_2,a_3,\ldots,a_m)$. Now,
$$
f(x,a_2,a_3,\ldots,a_m)=f(\pi_0(x),a_2,a_3,\ldots,a_m)+f(\pi_1(x),a_2,a_3,\ldots,a_m),
$$
is the decomposition as a sum of homogeneous element. It means that $y=\pi_j(f(x,a_2,a_3,\ldots,a_m))$ (for some $j\in\{0,1\}$), and $f(x,a_2,a_3,\ldots,a_m)\in I$, since $I$ is an ideal. Thus, the restriction of the projections are onto maps.
\end{proof}
From now on, we shall fix an ungraded ideal $I\subseteq\mathcal{A}$, such that $I\ne0$ and $I\ne\mathcal{A}$.

\begin{Lemma}\label{lem2}
Let $f=f(x_1,\ldots,x_m)\in\mathbb{F}_\Omega\langle X\rangle$ be a multilinear polynomial, where $m\ge2$, and $a_3$, \dots, $a_m\in\mathcal{A}$ be $\mathbb{Z}_2$-homogeneous elements. Denote
$$
d:(x,y)\in\mathcal{A}\times\mathcal{A}\mapsto f(x,y,a_3,\ldots,a_m)\in\mathcal{A}.
$$
Then, given $x$, $y\in I$, one has
$$
d(\pi_0(x),\pi_1(y))=d(\pi_1(x),\pi_0(y)),\quad d(\pi_0(x),\pi_0(y))=d(\pi_1(x),\pi_1(y)).
$$
\end{Lemma}
\begin{proof}
On one hand, $d(\pi_0(x),y)$, $d(x,\pi_0(y))\in I$. Moreover,
$$
d(\pi_0(x),y)=d(\pi_0(x),\pi_0(y))+d(\pi_0(x),\pi_1(y))
$$
is the decomposition of $d(\pi_0(x),y)$ as a sum of homogeneous elements. On the other hand,
$$
d(x,\pi_0(y))=d(\pi_0(x),\pi_0(y))+d(\pi_1(x),\pi_0(y))
$$
is the decomposition of $d(x,\pi_0(y))$ as the sum of its homogeneous components. One of the homogeneous component is the same for both elements. From \Cref{lem1}, it implies that $d(\pi_0(x),y)=d(x,\pi_0(y))$, since both belong to $I$. As a consequence, the other homogeneous component also coincide, that is,
$$
d(\pi_0(x),\pi_1(y))=d(\pi_1(x),\pi_0(y)).
$$
So we get the first relation.

Now,
$$
d(\pi_1(x),y)=d(\pi_1(x),\pi_0(y))+d(\pi_1(x),\pi_1(y)).
$$
Again, a homogeneous component of this element coincide with a homogeneous component of $d(x,\pi_0(y))$. Then, from \Cref{lem1}, both elements coincide so the other homogeneous component must be equal. It means that $d(\pi_0(x),\pi_0(y))=d(\pi_1(x),\pi_1(y))$.
\end{proof}

As a consequence, we have:
\begin{Lemma}\label{lem3}
Let $f=f(x_1,\ldots,x_m)\in\mathbb{F}_\Omega\langle X\rangle$ be a multilinear polynomial and $a_1$, \dots, $a_m\in I$. Then:
\begin{enumerate}
\renewcommand{\labelenumi}{(\roman{enumi})}
\item $\pi_0(f(a_1,\ldots,a_m))=2^{m-1}f(\pi_{i_1}(a_1),\ldots,\pi_{i_m}(a_m))$, for any sequence $(i_1,\ldots,i_m)$ containing an even number of $1$'s,
\item $\pi_1(f(a_1,\ldots,a_m))=2^{m-1}f(\pi_{i_1}(a_1),\ldots,\pi_{i_m}(a_m))$, for any sequence $(i_1,\ldots,i_m)$ containing an odd number of $1$'s.
\end{enumerate}
\end{Lemma}
\begin{proof}
First, from \Cref{lem2},
$$
f(\pi_{i_1}(a_1),\ldots,\pi_{i_m}(a_m))=f(\pi_{j_1}(a_1),\ldots,\pi_{j_m}(a_m))
$$
whenever $i_1+\cdots+i_m\equiv_2 j_1+\cdots+j_m$. Now, let $(i_1,\ldots,i_m)\in\{0,1\}^m$ be a sequence containing an even number of $1$'s. Writing $a_\ell=\pi_0(a_\ell)+\pi_1(a_\ell)$, one has, from multilinearity:
\begin{align*}
\pi_0(f(a_1,\ldots,a_m))&=\sum_{\substack{(j_1,\ldots,j_m)\\j_1+\cdots+j_m\equiv_20}}f(\pi_{j_1}(a_1),\ldots,\pi_{j_m}(a_m))\\&=2^{m-1}f(\pi_{i_1}(a_1),\ldots,\pi_{i_m}(a_m)).
\end{align*}
This proves (i). In a similar way we get (ii).
\end{proof}

Finally, we can prove the main result of this subsection:
\begin{proof}[Proof of \Cref{gradedsimplealg}]
If $\mathcal{A}$ is simple as an $\Omega$-algebra, then there is nothing to do. Otherwise, let $I\subseteq\mathcal{A}$ be an ungraded ideal, such that $I\ne0$ and $I\ne\mathcal{A}$. Denote $\Omega_{>0}=\cup_{n>0}\Omega_n$. As $I$ is itself an $\Omega_{>0}$-algebra, define
$$
\varphi:x+\varepsilon y\in I\otimes\mathbb{FZ}_2\mapsto2\pi_0(x)+2\pi_1(y)\in\mathcal{A}.
$$
Clearly $\varphi$ is a well-defined $\mathbb{Z}_2$-graded linear map, and \Cref{lem1} tells us that $\varphi$ is a linear isomorphism. It remains to show that $\varphi$ is a homomorphism of $\Omega_{>0}$-algebras. So, let $n>0$, $\omega\in\Omega_n$, and $\varepsilon^{i_1}a_1$, \dots, $\varepsilon^{i_n}a_n\in I\otimes\mathbb{FZ}_2$, and denote $i=i_1+\cdots+i_n$. Then, using \Cref{lem3}, one has
\begin{align*}
\varphi(\omega(\varepsilon^{i_1}a_1,\ldots,\varepsilon^{i_n}a_n))&=2\pi_i(\omega(a_1,\ldots,a_n))=2^n\omega(\pi_{i_1}(a_1),\ldots,\pi_{i_n}(a_n))\\%
&=\omega(2\pi_{i_1}(a_1),\ldots,2\pi_{i_n}(a_n))=\omega(\varphi(\varepsilon^{i_1}a_1),\ldots,\varphi(\varepsilon^{i_n}a_n)).
\end{align*}
Thus, $I\otimes\mathbb{FZ}_2\cong_2\mathcal{A}$ as $\mathbb{Z}_2$-graded $\Omega_{>0}$-algebras. Now, from \Cref{lem0}, we see that $I$ is a simple $\Omega_{>0}$-algebra. Moreover, for any $\mu\in\Omega_0$, if $\mu\in\mathcal{A}^0$, then we can define $\mu\mapsto\varphi^{-1}(\mu)\in I$. Thus, $I$ contains all $0$-ary operations of $\mathcal{A}$ of trivial homogeneous degree. Hence, $I$ is an $\Omega'$-algebra, and $I\otimes\mathbb{FZ}_2\cong_2\mathcal{A}$ is an isomorphism of $\Omega'$-algebras.
\end{proof}

\subsection{Some generalization}
Using the results of this section, we obtain the following generalizations of the counter-examples found in the associative case. We highlight these cases. First, we need a technical result. The associative case is proved in \cite[Theorem 7.2.3]{AGPR}. The same argument holds valid in the context of $\Omega$-algebras.
\begin{Lemma}
Let $\mathcal{A}$, $\mathcal{A}'$, $\mathcal{B}$, $\mathcal{B}'$ be finite-dimensional $\Omega$-algebras, and assume that $\mathrm{Id}_\Omega(\mathcal{A}')\subseteq\mathrm{Id}_\Omega(\mathcal{A})$, and $\mathrm{Id}_\Omega(\mathcal{B}')\subseteq\mathrm{Id}_\Omega(\mathcal{B})$. Then
$$
\mathrm{Id}_\Omega(\mathcal{A}'\otimes\mathcal{B}')\subseteq\mathrm{Id}_\Omega(\mathcal{A}\otimes\mathcal{B}).
$$
\end{Lemma}

As a consequence, combining the previous proposition and \Cref{fdgrassid}, we have an extension of \Cref{casemnfz2}.
\begin{Prop}
Let $\mathcal{A}$ be an $\Omega$-simple algebra, and let $m$, $m'\in\mathbb{N}$. Then, $\lfloor\frac{m}2\rfloor=\lfloor\frac{m}2\rfloor$ implies $\mathrm{Id}_\Omega(G_m(\mathcal{A}\otimes\mathbb{FZ}_2))=\mathrm{Id}_\Omega(G_{m'}(\mathcal{A}\otimes\mathbb{FZ}_2))$.\qed
\end{Prop}
\begin{Problem}
Does the converse of the previous corollary hold?
\end{Problem}

From \Cref{gradedsimplealg}, and repeating the reasoning of the associative case, we have an extension of \Cref{g1_ordinary}.
\begin{Prop}
Let $\mathcal{A}$ be a finite-dimensional simple $\Omega$-algebra. Then $\mathcal{A}\otimes\mathbb{FZ}_2=\mathcal{A}\oplus\varepsilon\mathcal{A}$ is a $\mathbb{Z}_2$-graded-simple $\Omega$-algebra. Moreover, for any $m\in\mathbb{N}$,
$$
\mathrm{Id}(G_1(\mathcal{A}\otimes\mathbb{FZ}_2))=\mathrm{Id}(G_m(\mathcal{A}))=\mathrm{Id}(G(\mathcal{A})),
$$
but $\mathcal{A}\otimes\mathbb{FZ}_2\not\cong\mathcal{A}$.\qed
\end{Prop}

\end{document}